\documentclass[MSNbibl,number,citesort,dvips]{arxbj}
\usepackage{upgreek}

\aid{0}
\volume{20}
\issue{2}
\pubyear{2014}
\firstpage{623}
\lastpage{644}
\doi{10.3150/12-BEJ500} 

\makeatletter
\newcommand{\rrVert}{\Vert}
\newcommand{\rrvert}{\vert}
\newcommand{\llVert}{\Vert}
\newcommand{\llvert}{\vert}
\newtheorem{theorem}{Theorem}
\newtheorem{lemma}{Lemma}
\newremark{remark}{Remark}

\makeatother

\begin{document}
\begin{frontmatter}

\title{Statistical convergence of Markov experiments to diffusion limits}
\runtitle{Statistical convergence of Markov experiments}

\begin{aug}
\author[1]{\fnms{Valentin} \snm{Konakov}\thanksref{1}\ead[label=e1]{VKonakov@hse.ru}},
\author[2]{\fnms{Enno} \snm{Mammen}\corref{}\thanksref{2}\ead[label=e2]{emammen@rumms.uni-mannheim.de}} \and
\author[3]{\fnms{Jeannette} \snm{Woerner}\thanksref{3}\ead[label=e3]{jeannette.woerner@math.uni-dortmund.de}}
\runauthor{V. Konakov, E. Mammen and J. Woerner} 
\address[1]{Higher School of Economics, Pokrovskii Boulevard 11,
103012 Moscow, Russia. \\\printead{e1}}
\address[2]{Department
of Economics, University of Mannheim, L7,3-5, 68229 Mannheim,
Germany. \\\printead{e2}}
\address[3]{Technische Universit\"at Dortmund,
Fakult\"at f\"ur Mathematik,
Vogelpothsweg 87,
44227 Dortmund, Germany. \printead{e3}}
\end{aug}

\received{\smonth{1} \syear{2012}}
\revised{\smonth{11} \syear{2012}}

%
\begin{abstract}
Assume that one observes the $k$th, $2k$th$,\ldots,nk$th value of a
Markov chain $X_{1,h},\ldots ,X_{nk,h}$. That means we assume that a high
frequency Markov chain runs in the background on a very fine time
grid but that it is only observed on a coarser grid. This
asymptotics reflects a set up occurring in the high frequency
statistical analysis for financial data where diffusion
approximations are used only for coarser time scales. In this paper,
we show that under appropriate conditions the L$_1$-distance between
the joint distribution of the
Markov chain and the distribution of the discretized diffusion
limit converges to zero. The result implies that the LeCam
deficiency distance between the statistical Markov experiment and
its diffusion limit converges to zero. This result can be applied to
Euler approximations for the joint distribution
of diffusions observed at points $\Delta, 2 \Delta, \ldots , n\Delta$.
The joint distribution can be approximated by
generating Euler approximations at the points $\Delta k^{-1}, 2 \Delta
k^{-1}, \ldots , n\Delta$. Our result implies
that under our regularity conditions the Euler approximation is
consistent for $n \to\infty$ if $nk^{-2}\to0$.
\end{abstract}

%

\end{frontmatter}

\section{Introduction}\label{sec:Intro}
In this paper, we consider
approximations of the joint distribution of a partially observed Markov
chain by the law of a discretely observed diffusion. More precisely we
consider a Markov
chain $X_{1,h},\ldots ,X_{nk,h}$ with values at $nk$ time points. This time
points are equal to $h,2h,\ldots ,nkh$ where $h$ is a time interval that
converges to zero. We assume that this process is only observed at each
$k$th point, that is, at the time points $kh, 2kh,\ldots ,nkh$. That
means we assume that a high frequency Markov chain runs in the
background on a very fine time grid but that it is only observed on
a coarser grid. This asymptotics reflects a set up occurring in the
high frequency statistical analysis for financial data where
diffusion approximations are used for coarser time scales. For
the finest scale, discrete pattern in the price processes become
transparent that could not be modeled by diffusions. The joint
distribution of the observed values of the Markov chain is denoted
by $P_h$. We assume that this joint distribution can be approximated
by the distribution of $(Y^*_1,\ldots ,Y^*_n)$ where $Y^*_1,\ldots ,Y^*_n$ are the
values of a diffusion $Y$ on the equidistant grid $kh,2kh,\ldots ,nkh$,
that is, $Y(ikh) = Y^*_i$. The joint distribution of
$(Y^*_1,\ldots ,Y^*_n)$ is
denoted by $Q_h$.

In this paper, we show that
\[
\|P_h-Q_h\|_1 \to0
\]
under some regularity conditions if
\[
{n \over k} \to0.
\]
This result can be applied to the asymptotic study of Markov
experiments $(P_{h,\theta}\dvtx  \theta\in\Theta)$ where $\Theta$ is a
finite or infinite-dimensional parameter set. Suppose that for this
family of Markov chains our assumptions apply uniformly for $\theta\in
\Theta$. Then one gets that $\sup_{\theta\in\Theta}\|P_{h,\theta
}-Q_{h,\theta}\|_1 \to0$ where $Q_{h,\theta}$ is the distribution of
the discretized limiting diffusion. This implies that the Markov
experiment $(P_{h,\theta}\dvtx  \theta\in\Theta)$ and the diffusion
experiment $(Q_{h,\theta}\dvtx  \theta\in\Theta)$ are asymptotically
equivalent in the sense of Le Cam's statistical theory of asymptotic
equivalence of experiments. Asymptotic equivalence of nonparametric
experiments has been discussed in a series of papers starting with
\cite{bb2} and \cite{bb14}. Work of statistical
experiments that converge to diffusions include \cite{bb13,bb10,bb3,bb16,bb6,bb7}.
Recently, Reiss \cite{bb15} provided asymptotic equivalence of a stochastic volatility model
with microstructure noise to a Gaussian shift experiment and a
regression model whereas Buchmann and M\"{u}ller \cite{bb4} considered the
relation between GARCH and COGARCH in the framework of statistical
equivalence. Our result justifies approximating diffusion models for
high frequency financial processes that are observed on a coarser grid.
We also outline that the Markov experiment and its diffusion
approximation differ in first order if $n/k$ does not converge to zero.
Then skewness properties of the Markov chain do not vanish in first
order. For a related paper see \cite{bb8}.
They consider estimation of the intensity of a discretely observed compound
Poisson process with symmetric Bernoulli jumps. For this model, they discuss
limit experiments under different assumptions on the limit of the
difference between neighbored time points.

We only discuss Markov chains with continuous state space. The
distribution of Markov chains with discrete state space cannot be
approximated by the distribution of continuous diffusions. For
asymptotic equivalence of the experiments $(P_{h,\theta}\dvtx  \theta\in
\Theta)$ and $(Q_{h,\theta}\dvtx  \theta\in\Theta),$ one has to show
that there exist Markov kernels $K_n$ and $L_n$ with $\sup_{\theta\in
\Theta}\|K_n P_{h,\theta}-Q_{h,\theta}\|_1 \to0$ and $\sup_{\theta
\in\Theta}\| P_{h,\theta}-L_n Q_{h,\theta}\|_1 \to0$. We expect
that such results could be shown by using expansions for transition
densities of Markov random
walks. The approach of this paper is based on expansions developed in
\cite{bb12}. The latter paper only considers Markov
chains with continuous state space. To treat Markov random
walks, their approach has to be carried over to the case of discrete
state spaces.

\section{The main result}\label{sec:Main}

We consider a Markov chain $X_{l,h}$ in $\mathbb{R}$ that runs on very
fine time grid and has
the following form
%
\begin{equation}\label{eq:001}
X_{l+1,h}=X_{l,h}+m ( X_{l,h} ) h+\sqrt{h}
\xi_{l+1,h}, \qquad X_{0,h}=x_0\in\mathbb{R}, \qquad l=0,\ldots ,nk-1.
\end{equation}
The innovation sequence $ ( \xi_{l,h} )_{l=1,...,nk}$ is
assumed to satisfy the Markov assumption: the conditional
distribution of $\xi_{l+1,h}$ given the past
$X_{l,h}=x_{l},\ldots ,X_{0,h}=x_{0}$ depends only on the last value
$X_{l,h}=x_{l}$ and has a conditional density $q (
x_{l},\cdot )$. The conditional variance
corresponding to this density is denoted by $\sigma^2 (x_{l})$ and the
conditional $%
\nu$th order moment by $\mu_{\nu}(x_{l})$. The transition
densities of $ ( X_{r,h} )$ given $ ( X_{l,h} )$
are denoted by
$p_{h} ( rh-lh,x_l,\cdot )$.

In the following, $C$ denotes a finite strictly positive constant whose
meaning may vary from line to line. We make the following assumptions.
\begin{enumerate}[(A3)]
\item[(A1)] It holds that $\int_{\mathbb{R}}y q (
x,y ) \,\mathrm{d}y=0$ for $x\in\mathbb{R}$.

\item[(A2)] There exist positive constants $\sigma_{\star}$ and
$\sigma^{\star}$ such that the variance $\sigma^2  (
x ) =\int_{\mathbb{R}}y^{2}q ( x,y )\, \mathrm{d}y$
satisfies
\[
\sigma_{\star}\leq\sigma^2 ( x ) \leq\sigma^{\star}
\]
for all $x\in\mathbb{R}$.

\item[(A3)] There exist a positive integer
$S^{\prime}> 1$ and a real nonnegative function $\psi ( y )$, $y\in\mathbb{R}$ satisfying $\sup_{y\in\mathbb{R}}\psi (
y ) <\infty$ and
$%
\int_{\mathbb{R}}\llvert  y\rrvert^{S}\psi ( y )
\,\mathrm{d}y<\infty$ with $S=2 S^{\prime}+4$ such that
\[
\bigl\llvert D_{y}^{\nu}q ( x,y ) \bigr\rrvert \leq\psi ( y
) ,\qquad  x,y\in\mathbb{R}, 0 \leq\nu \leq4.
\]
Moreover, for all $x,y\in R$, $j\geq
1 $
\[
\bigl\llvert D_{x}^{\nu}q^{(j)} ( x,y ) \bigr
\rrvert \leq Cj^{-1/2}\psi \bigl( j^{-1/2}y \bigr) ,\qquad 0 \leq\nu\leq3
\]
for a constant $C<\infty$. Here, $q^{(j)}(x,y)$ denotes the usual
$j$-fold convolution of $q$ for fixed $x$ as a function of $y$:
\[
q^{(j)}(x,y)=\int q^{(j-1)}(x,u)q(x,y-u)\,\mathrm{d}u,
\]
$q^{(1)}(x,y)=q(x,y)$.
\end{enumerate}
Note that the last
condition is very weak. It is motivated by
(A2) and the classical local limit theorem.
\begin{enumerate}[(A5)]
\item[(A4)] The functions $m ( x ) $ and
$\sigma ( x ) $ and their derivatives up to the order six are
continuous and bounded.
Furthermore, $%
D_{x}^{6 }\sigma ( x ) $ is H\"older continuous of order $0
< \alpha< 1$.

\item[(A5)] There exists $\varkappa
<\frac{1}{5}$ and a constant $C>0$ such that
\[
C^{-1}k^{-\varkappa}< h k < C.
\]
\end{enumerate}

The Markov chain $X_{l,h}$, see (\ref{eq:001}), is an approximation
to the
following stochastic differential equation in $\mathbb{R}$:%
%
\begin{equation}\label{a1}
\mathrm{d}Y_{s}=m ( Y_{s} ) \,\mathrm{d}s+\sigma ( Y_{s} )
\,\mathrm{d}W_{s},\qquad  Y_{0}=x_0\in\mathbb{R}, \qquad s\in[0,T],
\end{equation}
where $ ( W_{s} )_{s\geq0}$ is the standard Wiener
process. The conditional density of $Y_{t},$ given $Y_{s}=x$ is
denoted by $p ( t-s,x,\cdot ) $. We also write $Y(s)$ for
$Y_s$. The joint distribution of
$Y$ on the equidistant grid $kh,2kh,\ldots ,nkh$ is denoted by $Q_h$.

Our main result is stated in the following theorem.
%
\begin{theorem}
\label{theo1} Assume \textup{(A1)}--\textup{(A5)} and $nk^{-1}
\rightarrow0$. Then it holds that $\Vert P_{h} -
Q_{h}\Vert_{1} \rightarrow0$.
\end{theorem}
\begin{remark}
Theorem~\ref{theo1} can be generalized to higher
dimensions and to the nonhomogenous case. We only treat the univariate
homogenous case for simplicity. In our proof, we make use of the
representation (\ref{a2}) from \cite{bb5}
that is only available for the univariate case. For multivariate
reducible diffusions, one can apply the Hermite expansion given in \cite{bb1}.
\end{remark}
\begin{remark}
 The assumptions of Theorem~\ref{theo1} allow to
apply second order expansions for the transition densities of Markov
chains that have been developed in \cite{bb12}. In the
proof of Theorem~\ref{theo1}, we make only use of first order
expansions. For this reason, the assumptions could be weakened. For
example, we expect that one needs only four derivatives in (A4) instead
of six. We do not pursue this here because we will need the second
order expansions for getting the results in the following theorem.
\end{remark}
%
\begin{theorem}
\label{theo2} Assume \textup{(A1)}--\textup{(A5)}, $nh^{1+\delta}
\rightarrow0$ and $nk^{-2}
\rightarrow0$, where $\delta> 0$ is chosen such that the statement of
Theorem~\ref{theo4} holds for this choice. Suppose that the third
conditional moment $\mu_3(x)$ of innovations of the Markov chain
fulfills $\mu_3(x) \equiv0$. Then it holds that
$\Vert P_{h} -
Q_{h}\Vert_{1} \rightarrow0$.
\end{theorem}
\begin{remark}
 This result can be applied to Euler approximations
of diffusions and to Markov chains with symmetric innovations.
For Euler schemes that approximate the joint density of a diffusion at
points $\Delta, 2 \Delta, \ldots , n\Delta$ it means that one has to
generate Euler approximations of the diffusions at points $\Delta
k^{-1}, 2 \Delta k^{-1}, \ldots , n\Delta$ where $k \rightarrow\infty$
is chosen such that $nk^{-2}
\rightarrow0$ and $n(\Delta/k)^{1 + \delta}
\rightarrow0$. The joint distribution of the Euler values at the
points $\Delta, 2 \Delta, \ldots , n\Delta$ is then the approximation of
the joint distribution of the diffusion at these points. Under the
regularity assumptions of Theorem~\ref{theo2}, the Euler approximation
is consistent.
A more detailed discussion of the necessity of the above assumptions on
$k$ will be given elsewhere.
\end{remark}

We now show that our assumption on the growth of $k$ in Theorem~\ref
{theo1} is sharp. For this purpose, we consider a simple model of
Markov chains that converge to a Gaussian process and we show that for
this case $\Vert P_{h} -
Q_{h}\Vert_{1}$ does not converge to zero if the condition on the growth
of $k$ in Theorem~\ref{theo1} is not met.
%
\begin{theorem}
\label{theo3} Assume \textup{(A1)}--\textup{(A5)} for Markov chains with $m(x) \equiv1$
and innovation density $q(x, \cdot) = q(\cdot)$ not depending on $x$.
We assume that $nk^{-1}
\rightarrow c$ for a constant $c \not= 0$. Furthermore, suppose, that
$\mu_3(x)=\mu_3 \not= 0$ and that $kh \to0$. Then $\Vert P_{h} -
Q_{h}\Vert_{1}$ does not converge to zero.
\end{theorem}

\section{Proofs}\label{sec:proofmain}

The proof of Theorem~\ref{theo1} will be divided into several lemmas.
For the
proof, we will make use of the results in \cite{bb12}
where Edgeworth type expansions of $p_{h}$ were given for
nonhomogenous Markov chains in $\mathbb{R}^d$ for $d \geq1$. We now
restate their main result for
one-dimensional homogenous Markov chains. To formulate their result, we
need some
additional notation.

We will use the following differential operators $L$ and
$\widetilde{L}$:
%
%
\begin{eqnarray}\label{eq:001a}
Lf(t,x,y)&=&\frac{1}{2}\sigma^2 (x)\frac{\partial^{2}f(t,x,y)}{(\partial x)^2}+m(x)
\frac{%
\partial f(t,x,y)}{\partial x},\nonumber
\\[-8pt]\\[-8pt]
\widetilde{L}f(t,x,y)&=&\frac{1}{2}\sigma^2 (y)\frac{%
\partial^{2}f(t,x,y)}{(\partial x)^2}%
+m(y)\frac{\partial f(t,x,y)}{\partial
x}.\nonumber
\end{eqnarray}
We also need the following convolution type binary operation $\otimes$:
\[
f\otimes g ( t,x,y ) =\int_{0}^{t}\mathrm{d}u\int
_{R}f ( u,x,z ) g ( t-u,z,y ) \,\mathrm{d}z.
\]
We now introduce the following
differential operators
\begin{eqnarray*}
\mathcal{F}_{1}[f](t,x,y)&=&\frac{\mu_{3 }(x)}{6}D_{x}^{3 }f(t,x,y),
\\
\mathcal{F}_{2}[f](t,x,y)&=&\frac{\mu_{4 }(x)-3\sigma^4(x)}{24}D_{x}^{4}f(t,x,y).
\end{eqnarray*}
The Gaussian transition densities
$\widetilde{%
p}(t,x,y)$ are defined as
\begin{eqnarray*}
\widetilde{p} ( t,x,y ) &=& ( 2\uppi )^{-1/2} \sigma ( y )^{-1}
t^{-1/2} \exp \biggl( -\frac
{1}{2t} \bigl( y-x-tm ( y )
\bigr)^{2}\sigma ( y )^{-2} \biggr) .
\end{eqnarray*}
We are now in the position to state the Edgeworth type expansion for
Markov chain transition densities from \cite{bb12}.
%
\begin{theorem}[(\cite{bb12})]\label{theo4}
Assume \textup{(A1)}--\textup{(A5)}.
Then there exists a
constant %
$\delta>0$ such that the following expansion holds:
\begin{eqnarray*}
&&\sup_{x,y\in R} (kh)^{1/2} \biggl( 1+\biggl\llvert
\frac{y-x}{\sqrt{kh}}%
\biggr\rrvert^{S^{\prime}} \biggr) \\
&&\quad {}\times \bigl|
p_{h}(kh,x,y)-p(kh,x,y)
-h^{1/2}\pi_{1}(kh,x,y)-h\pi_{2}(kh,x,y) \bigr| =\mathrm{O}
\bigl(h^{1+\delta}\bigr),
\end{eqnarray*}
where S$^\prime$ is defined in Assumption \textup{(A3)} and
where
\begin{eqnarray*}
\pi_{1}(t-s,x,y)&=&\bigl(p\otimes\mathcal{F}_{1}[p]\bigr)
(t-s,x,y),
\\
\pi_{2}(t-s,x,y)&=&\bigl(p\otimes\mathcal{F}_{2}[p]\bigr)
(t-s,x,y)+p\otimes \mathcal{F}_{1}\bigl[p\otimes\mathcal{F}_{1}[p]
\bigr](t-s,x,y)
\\
&&{} +\tfrac{1}{2}p\otimes\bigl(L_{\star
}^{2}-L^{2}
\bigr)p(t-s,x,y).
\end{eqnarray*}
Here the operator L$_{\star}$ is
defined as $\widetilde{L}$, but with the coefficients
``frozen'' at the point $x$, that is,
\[
{L}^2_{\star}f(t,x,y)=\frac{1}{4}\sigma^4
(x)\frac{%
\partial^{4}f(t,x,y)}{(\partial x)^4}+ \sigma^2 (x)m(x)\frac{%
\partial^{3}f(t,x,y)}{(\partial x)^3}
+m(x)^2\frac{\partial^2 f(t,x,y)}{(\partial
x)^2}.
\]
\end{theorem}

We will apply this theorem for transition densities over the interval
$(ikh, (i+1) kh]$. The expansion of the theorem holds uniformly over $0
\leq i \leq n-1$.

We denote now the signed measure on $\mathbb{R}^{n}$ defined by the
products of $p + h^{1/2}\pi_{1}$ as $Q_{h}^{1}$ and the signed measure
defined by the products of $p + h^{1/2}\pi_{1} + h\pi_{2}$ as $Q_{h}^{2}$.
\begin{pf*}{Proof of Theorem~\ref{theo1}}
Theorem~\ref{theo1} immediately follows from the following two
lemmas.
\end{pf*}

In all lemmas of this section, we make the assumptions of
Theorem~\ref{theo1}.
%
\begin{lemma}\label{lemma3}
It holds that:
\begin{eqnarray*}
\bigl\Vert Q_{h}^{1} - Q_{h}\bigr\Vert_{1} =
\mathrm{o}(1) \qquad \mbox{for } n \to \infty.
\end{eqnarray*}
\end{lemma}
%
\begin{lemma}\label{lemma1}
It holds that:
\begin{eqnarray*}
\bigl\Vert P_{h} - Q_{h}^{1}\bigr\Vert_{1} =
\mathrm{o}(1) \qquad \mbox{for } n \to\infty.
\end{eqnarray*}
\end{lemma}

The hard part of these two lemmas is the proof of Lemma~\ref{lemma3}. For
the proof of the two lemmas, we will use a series of lemmas that are
stated and proved now. We will come back to the proofs of Lemmas \ref{lemma3} and \ref{lemma1} afterwards.

In our proofs, we make use of the following representation of
transition densities.
For the transition density $p ( t-s,x,\xi ) $
of the
diffusion (\ref{a1}), the following formula holds, see formula (3.2)
in \cite{bb5}
%
\begin{eqnarray}\label{a2}
p ( t-s,x,y ) &=&\widehat{p} ( t-s,x,y )\nonumber
\\[-8pt]\\[-8pt]
&& {}\times E\exp \biggl[ (t-s)\int_{0}^{1}g
\bigl[z_{\delta} \bigl( S(x),S(y ) \bigr)+%
\sqrt{(t-s)}B_{\delta}
\bigr]\,\mathrm{d}\delta \biggr] ,
\nonumber
\end{eqnarray}
where for $0\leq\delta\leq1$
$B_{\delta}$ is a Brownian bridge. Furthermore, for $u \geq0$ we put
$g(u)=-\frac{1}{2} ( C^{2}(u)+C^{\prime}(u) ) $ and
$z_{\delta
}(x,y)=(1-\delta)x+\delta y$ with
%
\begin{eqnarray}\label{a3}
\widehat{p} ( t-s,x,y ) &=&\frac{1}{\sqrt{2\uppi
(t-s)}\sigma(y)}\exp%
\biggl[ -
\frac{(S(y)-S(x))^{2}}{2(t-s)}+H(y)-H(x) \biggr] ,
\\
S(x)&=&\int_{0}^{x}\frac{\mathrm{d}u}{\sigma(u)},
\nonumber
\\[-8pt]\\[-8pt]
H(x)&=&\int_{0}^{S(x)}C(u)\,\mathrm{d}u \qquad \mbox{with } C(u)=
\frac{m(u)}{\sigma(u)}-\frac{1}{2}\sigma^{\prime}(u)\nonumber
\end{eqnarray}
for $ x,y, s,t \in\mathbb{R}$.

Note that under
our assumptions $g$ is bounded, $\llvert  g ( x ) \rrvert
\leq M,$ and, hence, for $t-s\leq kh$%
%
\begin{equation}\label{a4}
E\exp \biggl[ (t-s)\int_{0}^{1}g
\bigl[z_{\delta} \bigl( S(x),S(\xi ) \bigr)+\sqrt{%
(t-s)}B_{\delta}\bigr]\,\mathrm{d}\delta \biggr] \leq\exp[Mkh]\leq C^*
\end{equation}
for some constant $C^* > 0$ because of (A5).
For
the proof of Lemma~\ref{lemma3} we make use of the following lemmas.
These lemmas make use of some further technical lemmas, given in Section~\ref{sec4} that bound $\delta_1(x,y) = \sqrt{h} {\pi_{1} (kh, x ,
y)}/{p (kh,x,y)}$, $\delta_{2}(x,y)=h{\pi_{2}(kh,x,y)}/{p(kh,x,y)}$
and partial derivatives of the transition densities.
%
\begin{lemma}\label{lemma6a}
Put $\Delta_{i} = \delta_1(Y((i-1)kh), Y(ikh))$. Then we have for all
$p \geq1$ that under~$Q_{h}$
\[
\sup_{1 \leq i \leq n} E_{Q_{h}} |\Delta_i|^p \leq
C_p k^{-p/2}
\]
for some constants $C_p$ depending on $p$.
\end{lemma}
\begin{pf*}{Proof of Lemma~\ref{lemma6a}}
This lemma directly follows from Lemma~\ref{lemma4} and the representation
(\ref{a2}). Using these results, the moments of $\Delta_i$ can be
easily bounded by Gaussian moments.
\end{pf*}

Lemma~\ref{lemma6a} implies that for all $\rho> 0$ under
$Q_{h}$
\begin{eqnarray*}
\sup_{1 \le i \le n} |\Delta_{i}| = \mathrm{O}_{p}
\bigl(k^{-1/2} n^{\rho}\bigr).
\end{eqnarray*}
This bound would suffice for our purposes but for completeness we state
the following sharper bound that follows (from our Lemma~\ref{lemma4}
and) from Theorem~1 in \cite{bb9}, where bounds for moments for the modulus of continuity of
diffusions are given.
%
\begin{lemma}\label{lemma6}
We have that under
$Q_{h}$ that
\begin{eqnarray*}
\sup_{1 \le i \le n} |\Delta_{i}| = \mathrm{O}_{p}
\bigl(k^{-1/2} (\log n)^{3/2}\bigr).
\end{eqnarray*}
\end{lemma}
We now state a result on the order of sums of $\Delta_i$'s.
%
\begin{lemma}\label{lemma7}
Under $Q_{h}$ it holds that
\begin{eqnarray*}
\sum_{i=1}^{n} \Delta_{i} =
\mathrm{O}_{p} \biggl( \sqrt{\frac{n}{k}} \biggr).
\end{eqnarray*}
\end{lemma}
\begin{pf}
We have that $E_{Q_{h}}[\Delta_{i} \Delta_{j}] = 0$ for $i \ne j$
because the definition of $\Delta_{i}$ implies that $E_{Q_{h}}[\Delta_{i}|Y_{j}\dvtx  j \le i-1] = 0$. Thus, it holds
\begin{eqnarray*}
E_{Q_{h}} \Biggl[\Biggl(\sum_{i=1}^{n}
\Delta_{i}\Biggr)^{2}\Biggr] = \sum
_{i=1}^{n} E_{Q_{h}}\bigl[(
\Delta_{i})^2\bigr] = \mathrm{O}\bigl(n k^{-1} \bigr).
\end{eqnarray*}
This follows from Lemma~\ref{lemma6a}.
\end{pf}

Put $A_{n} =  \{ (Y(kh),\ldots ,Y(nkh))\dvtx \sup_{1 \le i \le n} |\Delta_{i}| \le\tau_{n}
\sqrt{\frac{n}{k}}, |\sum_{i=1}^{n} \Delta_{i}| \le\tau_{n}
\sqrt{\frac{n}{k}}  \}$, where $\tau_{n}
\rightarrow\infty$ with $\tau_{n} \sqrt{\frac{n}{k}}
\rightarrow0$.

Then we get from Lemmas \ref{lemma6}--\ref{lemma7} that
%
\begin{eqnarray}
\label{eqa1} Q_{h} (A_{n}) \rightarrow1.
\end{eqnarray}

For the proof of Lemma~\ref{lemma3}, we need the following additional
simple lemma.
%
\begin{lemma}\label{lemma9}
Consider the set $B_{n} = \{ x \in\mathbb{R}^{n}\dvtx  |x_{i}| \le\tau_{n} \sqrt{\frac{n}{k}}$ for $i=1,\ldots ,n; |\sum_{i=1}^n x_{i}| \le
\tau_{n} \sqrt{\frac{n}{k}} \} \subset\mathbb{R}^{n}$, where $\tau_{n}
\rightarrow\infty$ with $\tau_{n} \sqrt{\frac{n}{k}}
\rightarrow0$. Then it holds that
\begin{eqnarray*}
\sup_{x \in B_{n}} \Biggl\llvert 1 - \prod_{i=1}^{n}(1
+ x_{i})\Biggr\rrvert \rightarrow0 \qquad \mbox{for } n \to\infty.
\end{eqnarray*}
\end{lemma}

The lemma implies that
\[
\max_{1 \leq j \leq n} \sup_{x \in B_{nj}} \Biggl\llvert 1 - \prod
_{i=1}^{j}(1 + x_{i})\Biggr\rrvert
\rightarrow0 \qquad \mbox{for } n \to\infty,
\]
where $B_{nj} = \{ x \in\mathbb{R}^{j}\dvtx  |x_{i}| \le\tau_{n} \sqrt {\frac{n}{k}}$ for $i=1,\ldots ,j; |\sum_{i=1}^j x_{i}| \le\tau_{n}
\sqrt{\frac{n}{k}} \} \subset\mathbb{R}^{n}$. This follows by
putting $x_i=0$ for $i=j+1,\ldots ,n$.

The next lemma states that the expansion (\ref{eqa1}) also holds under
the measure $Q_{h}^{1}$.
%
\begin{lemma}\label{lemma8}
It holds that
\begin{eqnarray*}
\bigl|Q_{h}^{1}\bigr| (A_{n}) \rightarrow1.
\end{eqnarray*}
Here, $|Q_{h}^{1}|$ means the total variation measure of $Q_{h}^{1}$.
\end{lemma}
\begin{pf}
By application of Lemma~\ref{lemma9}, we get that
\begin{eqnarray*}
\bigl\vert \bigl|Q_{h}^{1}\bigr| (A_{n}) -
Q_{h} (A_{n}) \bigr\vert & \leq& \int I(A_n)
\bigl\llvert \mathrm{d}Q_{h}^{1} - \mathrm{d}Q_{h} \bigr\rrvert
\\
&=& \int\Biggl\llvert 1- \prod_{i=1}^n (1
+ \Delta_i)\Biggr\rrvert I(A_n) \,\mathrm{d} Q_h
\\
&=& \mathrm{o}(1).
\end{eqnarray*}
This implies the statement of the lemma because of (\ref{eqa1}).
\end{pf}

We now prove Lemma~\ref{lemma3}.
\begin{pf*}{Proof of Lemma~\ref{lemma3}}
We have that
\begin{eqnarray*}
\bigl\Vert Q_{h}^{1} - Q_{h}\bigr\Vert_{1} &=&
E_{Q_{h}} \Biggl[ \Biggl\llvert 1 - \prod_{i=1}^{n}
(1 + \Delta_{i})\Biggr\rrvert \Biggr]
\\
&=& E_{Q_{h}} \Biggl[ \Biggl\llvert 1 - \prod
_{i=1}^{n} (1 + \Delta_{i})\Biggr\rrvert I
(A_{n}) \Biggr] + \mathrm{o}(1),
\end{eqnarray*}
because of (\ref{eqa1}) and Lemma~\ref{lemma8}. Now the lemma
follows from Lemma~\ref{lemma9}.
\end{pf*}

It remains to prove Lemma~\ref{lemma1}.
\begin{pf*}{Proof of Lemma~\ref{lemma1}}
We can write $P_{h} = P_{h,1} \times\cdots\times P_{h,n}$ and
$Q_{h}^{1} = Q^{1}_{h,1} \times\cdots\times Q^{1}_{h,n}$ where
$P_{h,j}$, $Q^{1}_{h,j}$ are suitably defined (signed) Markov
kernels.
By using a telescope argument, we get with constants $C^*, C^{**} > 0$
that for $n$ large enough
\begin{eqnarray*}
\bigl\llVert P_{h}-Q_{h}^{1}\bigr
\rrVert_{1}&=&\int\bigl\llvert p_{h}(kh,x,z_{1})
\times\cdots \times p_{h}(kh,z_{n-1},z_{n})
\\
&& \hphantom{\int\bigl\llvert}{}-  \bigl(p+h^{1/2}\pi_{1}\bigr)
(kh,x,z_{1})\times\cdots\times\bigl(p+h^{1/2}\pi_{1}
\bigr) (kh,z_{n-1},z_{n})\bigr\rrvert \,\mathrm{d}z_{1}\cdots\, \mathrm{d}z_{n}
\\
&\leq& \int\bigl\llvert \bigl(p_{h}-p-h^{1/2}
\pi_{1}\bigr) (kh,x,z_{1})\bigr\rrvert \\
&&\hphantom{\int}{} \times p_{h}(kh,z_{1},z_{2})
\times\cdots\times p_{h}(kh,z_{n-1},z_{n})\,\mathrm{d}z_{1}\cdots\, \mathrm{d}z_{n}
\\
&&{} +\int\bigl|p+h^{1/2}\pi_{1}\bigr|(kh,x,z_{1})\bigl
\llvert \bigl(p_{h}-p-h^{1/2}\pi_{1}\bigr)
(kh,z_{1},z_{2})\bigr\rrvert
\\
&& \hphantom{{} +\int}{}\times p_{h}(kh,z_{2},z_{3})\times\cdots \times
p_{h}(kh,z_{n-1},z_{n})\,\mathrm{d}z_{1}\cdots\, \mathrm{d}z_{n}
\\
&& {}+\cdots +\int\bigl|p+h^{1/2}\pi_{1}\bigr|(kh,x,z_{1})
\times\cdots \times \bigl|p+h^{1/2}\pi_{1}\bigr|(kh,z_{n-2},z_{n-1})
\\
&& \hphantom{{}+\cdots +\int}{}\times \bigl\llvert \bigl(p_{h}-p-h^{1/2}
\pi_{1}\bigr) (kh,z_{n-1},z_{n})\bigr\rrvert
\,\mathrm{d}z_{1}\cdots \,\mathrm{d}z_{n}
\\
&\leq& \frac{C^{*}}{k} \Biggl( 1+ \sum_{j=1}^{n-1}E_{Q_h}
\Biggl[ \prod_{i=1}^{j}\bigl(\llvert
1+%
\Delta_{i}\rrvert \bigr) \Biggr] \Biggr)
\\
&\leq& \frac{C^{*}}{k} \Biggl( 1 + \sum_{j=1}^n
\bigl\| Q_{h,j}^1\bigr\|_1 \Biggr)
\\
&\leq& n \frac{C^{*}}{k} \bigl(1+ \mathrm{o}(1) \bigr)\leq\frac{C^{**}n}{k}%
=\mathrm{o}(1),\qquad n\rightarrow\infty,
\end{eqnarray*}
where
$\| Q_{h,j}^1\|_1 = \int| p + h^{1/2} \pi_1|(kh,x,z_1) \times\cdots
\times | p + h^{1/2} \pi_1|(kh,x,z_j) \,\mathrm{d}z_1\cdot \cdots \cdot dz_j$.
We used that
%
\begin{equation}
\label{lum2}\sup_x \int\bigl\llvert \bigl(p_{h}-p-h^{1/2}
\pi_{1}\bigr) (kh,x,z)\bigr\rrvert \,\mathrm{d}z \leq\frac{C^{*}}{k}
\end{equation}
for some constant $C^{*}> 0$ and that uniformly in $1 \leq j \leq n$,
$\| Q_{h,j}^1\|_1 \leq\| Q_{h,j}\|_1+ \| Q_{h,j}^1- Q_{h,j}\|_1 = 1 +
\| Q_{h,j}^1- Q_{h,j}\|_1 = 1 + \mathrm{o}(1)$. For the last equality, we used
Lemma~\ref{lemma3}.

From Theorem~\ref{theo4}, we get that the left-hand side of the
inequality can be bounded by:
%
\begin{equation}
\label{lum1}\sup_x \int\bigl\llvert h\pi_{2}(kh,x,z)
\bigr\rrvert \,\mathrm{d}z+ \mathrm{O}\bigl(h^{1+\delta} \bigr) \sup_x \int
(kh)^{-1/2} \biggl( 1+\biggl\llvert \frac{z-x}{\sqrt{kh}}%
\biggr
\rrvert^{S^{\prime}} \biggr)^{-1} \,\mathrm{d}z.
\end{equation}

According to Assumption (A5) $hk$ is bounded. Thus, the second term in
(\ref{lum1}) is of order $\mathrm{O}(h^{1+\delta} )=\mathrm{O}((hk)^{1+\delta}
k^{-1-\delta}) = \mathrm{O}(k^{-1-\delta} ) =\mathrm{O}(k^{-1} )$.

For the first term, we have the following bound from Lemma~\ref{lemma4b}:
\[
\mathrm{O}\bigl(k^{-1}\bigr) \sup_x \int p(kh,x,z) \biggl[ 1+
\biggl( \frac{%
\llvert  z-x\rrvert }{\sqrt{kh}} \biggr)^{7} \biggr] \,\mathrm{d}z.
\]
Now, the second factor of this bound is of order $\mathrm{O}(1)$ because of
(\ref{a2}). Thus, the bound is of order
$\mathrm{O}(k^{-1})$. This shows claim (\ref{lum2}) and concludes the proof of
the lemma.
\end{pf*}
\begin{pf*}{Proof of Theorem~\ref{theo2}}
It is enough to prove that
%
\begin{equation}\label{1}
\bigl\llVert Q_{h}-Q_{h}^{2}\bigr
\rrVert_{1}\rightarrow0,\qquad n\rightarrow \infty
\end{equation}

and%
%
\begin{equation}\label{2}
\bigl\llVert P_{h}-Q_{h}^{2}\bigr
\rrVert_{1}\leq Cnh^{1+\delta}.
\end{equation}

Claim (\ref{2}) can be shown with arguments similar to the ones used
in the proof of Lemma~\ref{lemma1}. Instead of the bound \ref{lum1},
one now uses the
expansion of Theorem~\ref{theo4}.

The proof of (\ref{1}) is
close to the proof of Lemma~\ref{lemma3}. With $\Delta^{(2)}_{i}=\delta_{2}(Y((i-1)kh),Y(ikh)),$ we obtain as it was done
before%
%
\begin{eqnarray}
\label{bound2a} \sup_{1\leq i\leq n}E\bigl\llvert \Delta^{(2)}_{i}
\bigr\rrvert^{p}&\leq& C_{p}k^{-p},
\\
\label{bound2b} \sup_{1\leq i\leq n}\bigl\llvert \Delta^{(2)}_{i}
\bigr\rrvert &=&\mathrm{O}_{p}\bigl(k^{-1} (\log n)^{7/2}
\bigr),
\\
\label{bound2c} \sum_{i=1}^{n}
\Delta^{(2)}_{i} &=&\mathrm{O}_{p}\biggl(\frac{\sqrt{n}}{k}
\biggr)
\end{eqnarray}
and the assertion of Theorem~\ref{theo2} follows with the same
arguments as used in the proof of Theorem~\ref{theo1}.
\end{pf*}
\begin{pf*}{Proof of Theorem~\ref{theo3}}
 Without loss of
generality, we assume that $\int x^2 q(x) \,\mathrm{d}x =1$.
Suppose that $\Vert P_{h} -
Q_{h}\Vert_{1}$ does converge to zero. This implies that the loglikelihood
$\log(\mathrm{d}P_h/\mathrm{d}Q_h)$ converges to zero in $Q_h$-probability. Thus, we
have that
\[
\sum_{i=1}^n \log\bigl(1 +
\Delta_i+ \Delta^{(2)}_{i}\bigr) \stackrel
{Q_h} {\rightarrow} 0.
\]
Note that the bounds (\ref{bound2a})--(\ref{bound2c}) remain valid
under the assumptions of Theorem~\ref{theo3}.
We now apply Lemma~\ref{lemma6} and (\ref{bound2b}). With a Taylor
expansion of the logarithm, we get from the last expression that
\[
\sum_{i=1}^n \Delta_i -
{1 \over2} \Delta_i^2 + \Delta^{(2)}_{i}
\stackrel{Q_h} {\rightarrow} 0.
\]
Because of (\ref{bound2c}) this shows that
%
\begin{eqnarray}
\label{var0} \sum_{i=1}^n
\Delta_i - {1 \over2} \Delta_i^2
\stackrel {Q_h} {\rightarrow} 0.
\end{eqnarray}
We will show that under $Q_h$
%
\begin{eqnarray}
\label{var1} \sum_{i=1}^n
\Delta_i \stackrel{d} {\rightarrow} N\bigl(0, \sigma^2
\bigr)
\end{eqnarray}
with $\sigma^2= 22 c \mu_3^2>0$ where $c$ is the limit of $n/k$.
Note that (\ref{var1}) contradicts (\ref{var0}) because these two
limit statements would imply that
\[
{1 \over2} \sum_{i=1}^n
\Delta_i^2 \stackrel{d} {\rightarrow} N\bigl(0,
\sigma^2\bigr).
\]
This is not possible because non negative random variables cannot
converge in distribution to a normal limit with strictly positive
variance. Thus for the statement of the theorem, it remains to prove
(\ref{var1}).

For the proof of (\ref{var1}), we will use a martingale central limit
theorem for the martingale $\sum_{j=1}^i \Delta_j$ with $\sigma
$-field ${\mathcal F}_{h,i}=\sigma(Y(0), Y(kh), \ldots ,Y(ikh))$.
According to Theorem~3.2 and Corollary~3.1 in \cite{bb11}, we
have for (\ref{var1}) to check that
%
\begin{eqnarray}
\label{var2}\sum_{i=1}^n E\bigl[
\Delta_i^2| {\mathcal F}_{h,i-1}\bigr] &\to&
\sigma^2,\qquad  \mbox{in probablity},
\\
\label{var2a}\max_{1 \leq i \leq n} \Delta_i^2 &\to&0, \qquad \mbox{in
probablity},
\\
\label{var2b}E\Bigl[\max_{1 \leq i \leq n} \Delta_i^2 \Bigr]&=& \mathrm{O}(1).
\end{eqnarray}
Claims (\ref{var2a})--(\ref{var2b}) follow directly from Lemmas \ref
{lemma6a}--\ref{lemma6}. Here, for the proof of (\ref{var2b}) one can
use the simple bound $\max_{1 \leq i \leq n} \Delta_i^2 \leq\sum_{i=1}^n \Delta_i^2$. Thus for the statement of the theorem it only
remains to prove (\ref{var2}).

For the limiting diffusion $Y_s$ we get that $\mathrm{d}Y_s = \mathrm{d}s + \mathrm{d}W_s$. For
this case, it holds that $p(s,t,x,y) = \widehat p(s,t,x,y) = (2 \uppi
(t-s))^{-1/2} \exp(-(y-x-(t-s))^2 (2(t-s))^{-1})$. We now give an
estimate for
%
\begin{eqnarray}
\label{var3} 6 h^{-1/2} \delta_1(x,y) p(kh,x,y) =
\mu_3 \int_0^{kh} \mathrm{d}u \int p(u,x,\xi)
{\partial^3 \over\partial\xi^3} p(kh-u,\xi,y) \,\mathrm{d}\xi.
\end{eqnarray}
Calculations close to the proof of (\ref{a7d}) give the following
estimate with a constant $C>0$:
\[
\biggl\llvert {\partial^3 p(t,x,y) \over\partial x^3} + {\partial^3
p(t,x,y) \over\partial y^3}\biggr\rrvert
\leq{C p(t,x,y) \over\sqrt t} \biggl( 1 + \biggl\llvert { y-x\over\sqrt t}
\biggr\rrvert^3 \biggr).
\]
Using this estimate in (\ref{var3}), we obtain
\[
6 h^{-1/2} \delta_1(x,y) p(kh,x,y) = I(x,y)+ \mathit{II}(x,y),
\]
where
%
\begin{eqnarray}\label{var3.1}
\mathit{II}(x,y)&=& - \mu_3 \int_0^{kh} \mathrm{d}u
\int p(u,x,\xi) {\partial^3 \over\partial y^3} p(kh-u,\xi,y) \,\mathrm{d}\xi
\nonumber
\\
&=& - \mu_3 {\partial^3 \over\partial y^3} \int_0^{kh}
\mathrm{d}u \int p(u,x,\xi) p(kh-u,\xi,y) \,\mathrm{d}\xi
\nonumber
\\
&=& - \mu_3 kh {\partial^3 \over\partial y^3} p(kh,x,y)
\\
&=& - \mu_3 { p(kh,x,y) \over\sqrt{kh}} \biggl( \biggl( \sqrt{kh} -
{y-x \over \sqrt{kh}} \biggr)^3- 2 \biggl( \sqrt{kh} -
{y-x \over
\sqrt{kh}} \biggr)^2 - \biggl( \sqrt{kh} -
{y-x \over \sqrt{kh}} \biggr) \biggr)
\nonumber
\\
&=& - \mu_3 { p(kh,x,y) \over\sqrt{kh}} Q_3 \biggl(
\sqrt{kh} - {y-x
\over \sqrt{kh}} \biggr)\nonumber
\end{eqnarray}
with $Q_3(z) = z^3 -2z^2-z$. For the term $I(x,y),$ we have the
following bound with a (new) constant $C>0$:
\[
\bigl|I(x,y)\bigr| \leq C \int_0^{kh} \mathrm{d}u \int
{p(u,x,\xi) p(kh-u,\xi,y) \over
\sqrt{kh-u}} \biggl( 1 + \biggl\llvert { y-x\over\sqrt{kh-u}}\biggr
\rrvert^3 \biggr) \,\mathrm{d}\xi.
\]
Using the same substitution as in the proof of Lemma~\ref{lemma4}, we
obtain the following estimate
%
\begin{eqnarray}\label{var4}
\bigl|I(x,y)\bigr| \leq C \sqrt{kh} { p(kh,x,y) } P_3 \biggl(\biggl\llvert
{y-x \over \sqrt{kh}} \biggr\rrvert \biggr)
\end{eqnarray}
for a polynomial $P_3(z) $ of degree 3 with positive coefficients. Put now
\begin{eqnarray*}
\Delta_{1,i} &=& {h^{1/2} I(Y((i-1)kh), Y(ikh)) \over 6 p(kh,
Y((i-1)kh), Y(ikh))},
\\
\Delta_{2,i} &=& {h^{1/2} \mathit{II}(Y((i-1)kh), Y(ikh)) \over6 p(kh,
Y((i-1)kh), Y(ikh))}.
\end{eqnarray*}
We will show that
%
\begin{eqnarray}
\label{var4a}\sum_{i=1}^n E\bigl[
\Delta_{1,i}^2| {\mathcal F}_{h,i-1}\bigr] &\to&0,
\qquad \mbox{in probablity},
\\
\label{var4b}\sum_{i=1}^n E\bigl[
\Delta_{2,i}^2| {\mathcal F}_{h,i-1}\bigr] &\to&
\sigma^2,\qquad  \mbox{in probablity}.
\end{eqnarray}
Because of $\Delta_{i} = \Delta_{1,i}+\Delta_{2,i}$ this shows (\ref{var2}).

We get from (\ref{var4}) with a new constant $C> 0$ that
%
\begin{eqnarray}\label{var4c}
E\bigl[\Delta_{1,i}^2| {\mathcal F}_{h,i-1}\bigr]
\leq C kh^2 E \biggl[P_3 \biggl( \biggl\llvert
{Y(ikh)-Y((i-1)kh) \over \sqrt {kh}}\biggr\rrvert \biggr)^2 \Bigl\vert {\mathcal
F}_{h,i-1} \biggr].
\end{eqnarray}

Conditionally given ${\mathcal F}_{h,i-1}$, $(Y(ikh)-Y((i-1)kh)) /
\sqrt{kh}$ has a normal distribution with mean $\sqrt{kh}$ and
variance 1. Because of $kh \to0$ (by assumption), we get that the
expectation on the right hand side of (\ref{var4c}) is uniformly
bounded, for $1 \leq i \leq n, n \geq1$. Furthermore, we have that $n
kh^2 = (n/k) (kh)^2 \to0$. Thus, (\ref{var4c}) implies (\ref{var4a}).

It remains to check (\ref{var4b}). For the proof of this claim, we
apply the explicit expression (\ref{var3.1}) and we get that
\begin{eqnarray*}
\sum_{i=1}^n E\bigl[\Delta_{2,i}^2|
{\mathcal F}_{h,i-1}\bigr] &=& {1 \over k}
\mu_3^2 \sum_{i=1}^n
E \biggl[Q_3 \biggl(\sqrt {kh} - {Y(ikh)-Y((i-1)kh) \over \sqrt{kh}}
\biggr)^2 \Bigl\vert {\mathcal F}_{h,i-1} \biggr]
\\
&=& {n \over k} \mu_3^2
{1 \over\sqrt{2 \uppi}} \int\bigl(z^6 + 2 z^4 +
z^2\bigr) \mathrm{e}^{-z^2/2}\, \mathrm{d}z
\\
&=& 22 {n \mu_3^2 \over k} .
\end{eqnarray*}
Now, because of $n/k \to c$, we get that the right-hand side of this
equation converges to $\sigma^2$. This concludes the proof.
\end{pf*}

\section{Some technical lemmas}\label{sec4}

This section collects some technical lemmas that were used in the
proofs of the last section. In all lemmas of this section, we make the
assumptions of
Theorem~\ref{theo1}.
%
\begin{lemma}\label{lemma5} For all $c > 0$ there exists a constant
$C> 0$ such that the following estimates hold for $0 \leq t-s \leq c$
%
\begin{eqnarray}
\label{a6a}\biggl\llvert \frac{\partial}{\partial x}p(t-s,x,y)\biggr\rrvert &\leq& C
\frac{%
p(t-s,x,y)}{\sqrt{t-s}}\biggl(\sqrt{t-s}+\frac{\llvert  y-x\rrvert
}{%
\sqrt{t-s}}\biggr),
\\
\label{a6}\biggl\llvert \frac{\partial}{\partial y}p(t-s,x,y)\biggr\rrvert &\leq& C
\frac{%
p(t-s,x,y)}{\sqrt{t-s}}\biggl(\sqrt{t-s}+\frac{\llvert  y-x\rrvert
}{%
\sqrt{t-s}}\biggr),
\\
\label{a7}\biggl\llvert \frac{\partial^{2}}{\partial y^{2}}p(t-s,x,y)\biggr\rrvert &\leq& C%
\frac{p(t-s,x,y)}{t-s}\biggl(1+\sqrt{t-s}+\frac{\llvert  y-x\rrvert }{%
\sqrt{t-s}}\biggr)^{2},
\\
\label{a7a}\biggl\llvert \frac{\partial^{2}}{\partial x^{2}}p(t-s,x,y)\biggr\rrvert &\leq& C%
\frac{p(t-s,x,y)}{t-s}\biggl(1+\sqrt{t-s}+\frac{\llvert  y-x\rrvert }{%
\sqrt{t-s}}\biggr)^{2},
\\
\label{a7b}\biggl\llvert \frac{\partial^{3}}{\partial x^{3}}p(t-s,x,y)\biggr\rrvert &\leq& C%
\frac{p(t-s,x,y)}{(t-s)^{3/2}}\biggl(1+\sqrt{t-s}+\frac{\llvert
y-x\rrvert }{%
\sqrt{t-s}}\biggr)^{3},
\\
\label{a7c}\biggl\llvert \frac{\partial^{4}}{\partial x^{2}\,\partial
y^{2}}p(t-s,x,y)\biggr\rrvert &\leq& C%
\frac{p(t-s,x,y)}{(t-s)^2}\biggl(1+\sqrt{t-s}+\frac{\llvert  y-x\rrvert }{%
\sqrt{t-s}}\biggr)^{4},
\\
\label{a7d}\biggl\llvert \biggl(\frac{\partial^{2}}{\partial x\,\partial y}- \frac
{\partial^{2}}{\partial x^{2}} \biggr)p(t-s,x,y)
\biggr\rrvert &\leq& C 
\frac{p(t-s,x,y)}{\sqrt{t-s}} \nonumber\\[-8pt]\\[-8pt]
&&{}\times\biggl(1+\frac{\llvert  y-x\rrvert }{\sqrt{t-s}}+
\biggl\llvert \frac{ y-x }{\sqrt{t-s}}\biggr\rrvert^2+\biggl\llvert
\frac{ y-x }{\sqrt{t-s}}\biggr\rrvert^3 \biggr).\nonumber
\end{eqnarray}
\end{lemma}
\begin{pf}
We prove the second, the
third and the last inequality. The remaining inequalities can be proved
exactly in the same way. From (\ref{a3}), we obtain%
%
\begin{eqnarray}
\label{a14}\frac{\partial}{\partial y}\widehat{p} ( t-s,x,y ) &=&-\frac{\sigma^{\prime}(y)}{\sqrt{2\uppi(t-s)}\sigma^{2}(y)}\exp \biggl[ -
\frac{%
(S(y)-S(x))^{2}}{2(t-s)}+H(y)-H(x) \biggr]
\nonumber
\\
&&{} +\frac{1}{\sqrt{2\uppi(t-s)}\sigma(y)}\exp \biggl[ -\frac
{(S(y)-S(x))^{2}}{%
2(t-s)}+H(y)-H(x) \biggr]
\nonumber
\\[-8pt]\\[-8pt]
&&\quad {} \times \biggl( H^{\prime}(y)-\frac{(S(y)-S(x))}{%
(t-s)\sigma(y)} \biggr)
\nonumber
\\
&=& \widehat{p} ( t-s,x,y ) \biggl[ -\frac{\sigma^{\prime
}(y)}{%
\sigma(y)}+H^{\prime}(y)-
\frac{S(y)-S(x)}{(t-s)\sigma(y)} \biggr] ,\nonumber
\\
\label{a15}\frac{\partial^{2}}{\partial y^{2}}\widehat{p} ( t-s,x,y ) &=& \frac{%
\partial}{\partial y}\widehat{p} (
t-s,x,y ) \biggl[ -\frac{\sigma^{\prime}(y)}{\sigma(y)}+H^{\prime}(y)-\frac{S(y)-S(x)}{(t-s)\sigma
(y)}%
\biggr]
\nonumber
\\
&&{} +\widehat{p} ( t-s,x,y ) \biggl[ \frac{ ( \sigma^{\prime
}(y) )^{2}-\sigma(y)\sigma^{\prime\prime}(y)}{\sigma^{2}(y)} 
+H^{\prime\prime}(y)\nonumber\\
&&\hphantom{{} +\widehat{p} ( t-s,x,y ) \biggl[}{}-\frac{1-\sigma^{\prime}(y)(S(y)-S(x))}{%
(t-s)\sigma^{2}(y)} \biggr]
\nonumber
\\[-8pt]\\[-8pt]
&=& \widehat{p} ( t-s,x,y ) \biggl[ -\frac{\sigma^{\prime
}(y)}{%
\sigma(y)}+H^{\prime}(y)-
\frac{S(y)-S(x)}{(t-s)\sigma(y)} \biggr]^{2}\nonumber
\\
&&{} +\widehat{p} ( t-s,x,y ) \biggl[ \frac{ ( \sigma^{\prime
}(y) )^{2}-\sigma(y)\sigma^{\prime\prime}(y)}{\sigma^{2}(y)} 
+H^{\prime\prime}(y)\nonumber
\\
&&\hphantom{{} +\widehat{p} ( t-s,x,y ) \biggl[}{}  -\frac{1}{(t-s)\sigma^{2}(y)}+\frac{(S(y)-S(x))}{(t-s)}\frac
{%
\sigma^{\prime}(y)}{\sigma^{2}(y)} \biggr].\nonumber
\end{eqnarray}
It follows from (\ref{a14}) and (\ref{a15}) and our assumptions that
%
\begin{eqnarray}
\label{a16}\biggl\llvert \frac{\partial}{\partial y}\widehat{p} ( t-s,x,y ) \biggr\rrvert &\leq& C
\frac{\widehat{p} ( t-s,x,y )
}{\sqrt{t-s}}%
\biggl( \sqrt{t-s}+\frac{\llvert  S(y)-S(x)\rrvert }{\sqrt {t-s}} \biggr) ,
\\
\label{a17}\biggl\llvert \frac{\partial^{2}}{\partial y^{2}}\widehat{p} ( t-s,x,y ) \biggr\rrvert &\leq& C
\frac{\widehat{p} (
t-s,x,y )
}{t-s} \biggl( 1+\sqrt{t-s}+\frac{\llvert  S(y)-S(x)\rrvert
}{\sqrt{t-s}}%
\biggr)^{2}.
\end{eqnarray}
It is easy to see that
%
\begin{eqnarray}
\label{a18}&&\biggl\llvert \frac{\partial}{\partial y}E\exp \biggl[ (t-s)\int_{0}^{1}g
\bigl[z_{%
\delta} \bigl( S(x),S(y ) \bigr)+\sqrt{(t-s)}B_{\delta}
\bigr]\,\mathrm{d}\delta \biggr] \biggr\rrvert
\nonumber
\\[-8pt]\\[-8pt]
&& \quad \leq C(t-s)E\exp \biggl[ (t-s)\int_{0}^{1}g
\bigl[z_{\delta} \bigl( S(x),S(y ) \bigr)+\sqrt{(t-s)}B_{\delta}
\bigr]\,\mathrm{d}\delta \biggr] ,\nonumber
\\
\label{a19}&& \biggl\llvert \frac{\partial^{2}}{\partial y^{2}}E\exp \biggl[ (t-s)\int_{0}^{1}g
\bigl[z_{\delta} \bigl( S(x),S(y ) \bigr)+\sqrt {(t-s)}B_{\delta
}
\bigr]\,\mathrm{d}\delta \biggr] \biggr\rrvert
\nonumber
\\[-8pt]\\[-8pt]
&& \quad \leq C(t-s)^2E\exp \biggl[ (t-s)\int_{0}^{1}g
\bigl[z_{\delta} \bigl( S(x),S(y ) \bigr)+\sqrt{(t-s)}B_{\delta}
\bigr]\,\mathrm{d}\delta \biggr].\nonumber
\end{eqnarray}
The second and the third inequality of the statement of the lemma now
follow from our assumptions and from (\ref{a2}), (\ref{a4}), (\ref
{a16})--(\ref{a19}).

It remains to show (\ref{a7d}). For a proof of this claim, note that
\begin{eqnarray*}
&&\biggl(\frac{\partial^{2}}{\partial x\,\partial y}- \frac{\partial^{2}}{\partial x^{2}} \biggr)\widehat p(t-s,x,y)\\
&&\quad = \widehat
p(t-s,x,y) \biggl[ \biggl(\frac{S(y) -S(x) }{(t-s) \sigma(x)} - H^{\prime}(x) \biggr)
\\
&&\hphantom{\quad = \widehat
p(t-s,x,y) \biggl[}{}\times \biggl(- \frac{\sigma^{\prime}(y)}{\sigma(y)} + H^{\prime
}(y)- H^{\prime}(x)+
\frac{(S(y) -S(x) )( \sigma^{-1}(x)-\sigma^{-1}(y))}{t-s} \biggr)
\\
&&\hphantom{\quad = \widehat
p(t-s,x,y) \biggl[}{}+ \biggl(- H^{\prime\prime}(x) - \frac{\sigma^{\prime
}(x)}{\sigma^2(x)}
\frac{(S(y) -S(x) )}{t-s} - \frac{1}{\sigma(x)} \frac{ \sigma^{-1}(x)-\sigma^{-1}(y)}{t-s} \biggr) \biggr] .
\end{eqnarray*}
Claim (\ref{a7d}) follows from our assumptions and (\ref{a2}).
\end{pf}

Put
\begin{eqnarray*}
\delta_1(x,y) = \sqrt{h} \frac{\pi_{1} (kh, x , y)}{p
(kh,x,y)}.
\end{eqnarray*}

We will also make use of the following bound.
%
\begin{lemma}\label{lemma4} There exists a constant C such that for
$x,y \in\mathbb{R}$
\[
\bigl\llvert \delta_1(x,y)\bigr\rrvert \leq\frac{C}{%
\sqrt{k}}
\biggl(1+\frac{\llvert  y-x\rrvert }{\sqrt{kh}}\biggr)^{3}.
\]
\end{lemma}
\begin{pf}
Note that by definition of
$\pi_{1}$:
%
\begin{eqnarray}\label{a5}
6h^{-1/2}\delta_1(x,y)p(kh,x,y)&=& \int_{0}^{kh}\mathrm{d}u
\int p ( u,x,\xi ) \mu_{3}(\xi)\frac{\partial^{3}}{\partial\xi^{3}}p ( kh-u,\xi ,y ) \,\mathrm{d}\xi
\nonumber
\\
&=&\int_{0}^{kh/2}\mathrm{d}u\cdots +\int_{kh/2}^{kh}\mathrm{d}u\cdots
\\
&\doteqdot& \Im_{1}+\Im_{2}.\nonumber
\end{eqnarray}
We now apply the
estimates of Lemma~\ref{lemma5} to obtain the upper bounds for $\Im_{1}$ and $\Im_{2}$ in (\ref%
{a5}). For $u\in\lbrack\frac{kh}{2},kh],$ we apply two times
integrations by parts. From our assumptions on $\mu_{3}(\xi)$ and
from (\ref{a4}), (\ref{a6}) and (\ref{a7}) we obtain that
%
\begin{eqnarray}\label{a8}
\llvert \Im_{2}\rrvert &=&\biggl\llvert \int_{kh/2}^{kh}\mathrm{d}u
\int \frac{%
\partial^{2}}{\partial\xi^{2}}\bigl[p ( u,x,\xi ) \mu_{3}(\xi)\bigr]
\frac{\partial}{\partial\xi}p ( kh-u,\xi ,y ) \,\mathrm{d}\xi\biggr\rrvert
\nonumber
\\
&\leq& \int_{kh/2}^{kh}\mathrm{d}u\int\biggl\llvert
\frac
{\partial^{2}}{\partial
\xi^{2}}\bigl[p ( u,x,\xi ) \mu_{3}(\xi )\bigr]\biggr|\biggl| \frac{\partial}{\partial\xi}p ( kh-u,\xi ,y ) \biggr\rrvert \,\mathrm{d}\xi
\nonumber
\\
&\leq& C\int_{kh/2}^{kh}\mathrm{d}u\int\frac{p ( u,x,\xi
) }{u}
\biggl( 1+\sqrt{u}+\frac{\llvert  S(\xi
)-S(x)\rrvert }{\sqrt{u}} \biggr)^{2}
\nonumber
\\[-8pt]\\[-8pt]
&&\hphantom{C\int_{kh/2}^{kh}\mathrm{d}u\int}{} \times\frac{p ( kh-u,\xi,y ) }{\sqrt{kh-u}} \biggl( \sqrt {kh-u}+%
\frac{\llvert  S(y)-S(\xi)\rrvert }{\sqrt{kh-u}}
\biggr) \,\mathrm{d}\xi
\nonumber
\\
&\leq& \frac{C}{kh}\exp[2Mkh] \int_{kh/2}^{kh}
\frac{\mathrm{d}u}{\sqrt{kh-u}}\int\widehat{p} ( u,x,\xi ) \widehat{p} ( kh-u,\xi,y )\biggl( 1+\sqrt{u}+\frac{\llvert  S(\xi)-S(x)\rrvert
}{%
\sqrt{u}} \biggr)^{2}
\nonumber
\\
&& \hphantom{\frac{C}{kh}\exp[2Mkh] \int_{kh/2}^{kh}\frac{\mathrm{d}u}{\sqrt{kh-u}}\int}{}\times  \biggl(
\sqrt{kh-u}+\frac{\llvert
S(y)-S(\xi
)\rrvert }{\sqrt{kh-u}} \biggr) \,\mathrm{d}\xi.\nonumber
\end{eqnarray}
For $u\in\lbrack0,\frac{kh}{2}]$ we get from (\ref{a4}), (\ref
{a6}) and (\ref{a7}) again by applying
integration by parts:
%
\begin{eqnarray}\label{a9}
\llvert \Im_{1}\rrvert &=&\int_{0}^{kh/2}\mathrm{d}u
\int\biggl\llvert \frac{\partial}{\partial\xi}\bigl[p ( u,x,\xi ) \mu_{3}(\xi)
\bigr]\biggr|\biggl| \frac{\partial^{2}}{\partial
\xi^{2}}%
p ( kh-u,\xi,y )
\biggr\rrvert \,\mathrm{d}\xi
\nonumber
\\
&\leq& C\int_{0}^{kh/2}\mathrm{d}u\int\frac{p ( u,x,\xi )
}{\sqrt{u}}
\biggl( \sqrt{u}+\frac{\llvert
S(\xi)-S(x)\rrvert }{\sqrt{u}} \biggr)
\nonumber
\\
&&\hphantom{C\int_{0}^{kh/2}\mathrm{d}u\int}{} \times\frac{p ( kh-u,\xi,y ) }{kh-u} \biggl( 1+\sqrt {kh-u}+\frac{%
\llvert  S(y)-S(\xi)\rrvert }{\sqrt{kh-u}}
\biggr)^{2}\,\mathrm{d}\xi
\\
&\leq& \frac{C}{kh}\exp[2Mkh] \int_{0}^{kh/2}
\frac{\mathrm{d}u}{\sqrt{u}}\int\widehat{p}%
( u,x,\xi ) \widehat{p} ( kh-u,\xi ,y
)
 \biggl( \sqrt{u}+\frac{\llvert  S(\xi)-S(x)\rrvert }{%
\sqrt{u}} \biggr) \nonumber\\
&& \hphantom{\frac{C}{kh}\exp[2Mkh] \int_{0}^{kh/2}
\frac{\mathrm{d}u}{\sqrt{u}}\int}{}\times\biggl( 1+\sqrt{kh-u}+
\frac{\llvert  S(y)-S(\xi
)\rrvert }{\sqrt{kh-u}} \biggr)^{2}\,\mathrm{d}\xi.\nonumber
\end{eqnarray}
We now use the following substitution:
\begin{eqnarray*}
u^{\prime}&=&kh-u,
\\
z(\xi)&=& \biggl( \frac{kh}{kh-u^{\prime}} \biggr)^{1/2}\frac
{(S(\xi)-S(y))}{%
\sqrt{u^{\prime}}}
\\
&& {}+ \biggl( \frac{u^{\prime}}{kh-u^{\prime}} \biggr)^{1/2}\frac{%
(S(y)-S(x))}{\sqrt{kh}}.
\end{eqnarray*}
Note that
%
\begin{eqnarray}
\label{a10}\mathrm{d}\xi&=& \bigl( kh-u^{\prime} \bigr)^{1/2}\bigl(u^{\prime}
\bigr)^{1/2} ( kh )^{-1/2}\sigma(\xi)\,\mathrm{d}z,
\\
\label{a11}z^{2}+\frac{(S(y)-S(x))^{2}}{kh}&=&\frac{kh}{(kh-u^{\prime})}\frac
{(S(\xi
)-S(y))^{2}}{(u^{\prime})}
\nonumber
\\
&&{} +\frac{(u^{\prime})}{(kh-u^{\prime})}\frac{(S(y)-S(x))^{2}}{kh}
\nonumber
\\
&& {}+2\frac{(S(\xi)-S(y))(S(y)-S(x))}{kh-u^{\prime}}+\frac
{(S(y)-S(x))^{2}}{kh%
}
\nonumber
\\[-8pt]\\[-8pt]
&=&\frac{(S(\xi)-S(y))^{2}}{u^{\prime}}+\frac{(S(\xi
)-S(y))^{2}}{%
kh-u^{\prime}}\nonumber\\
&&{}+2\frac{(S(\xi)-S(y))(S(y)-S(x))}{kh-u^{\prime}}
 +\frac{(S(y)-S(x))^{2}}{kh-u^{\prime}}
\nonumber
\\
&=&\frac{(S(\xi
)-S(y))^{2}}{u^{\prime
}}+\frac{(S(\xi)-S(x))^{2}}{kh-u^{\prime}} .\nonumber
\end{eqnarray}
From (\ref{a10}) and (\ref{a11}), we get that
\begin{eqnarray}\label{a12}
\llvert \Im_{1}\rrvert &\leq& \frac{C}{kh}\exp[2Mkh]\exp \bigl[
H(y)-H(x) \bigr] \nonumber\\
&&{}\times\int_{kh/2}^{kh}\frac{\mathrm{d}u^{\prime}}{\sqrt{%
kh-u^{\prime}}}
\int\frac{1}{\sqrt{2\uppi(kh-u^{\prime})}\sigma(\xi)}\frac
{1}{%
\sqrt{2\uppi(u^{\prime})}\sigma(y)}
\nonumber
\\
&& \hphantom{{}\times\int_{kh/2}^{kh}\frac{\mathrm{d}u^{\prime}}{\sqrt{%
kh-u^{\prime}}}
\int}{}\times\exp \biggl[ -\frac{(S(\xi)-S(x))^{2}}{2(kh-u^{\prime
})}-\frac{%
(S(y)-S(\xi))^{2}}{2(u^{\prime})} \biggr] \nonumber\\
&&\hphantom{{}\times\int_{kh/2}^{kh}\frac{\mathrm{d}u^{\prime}}{\sqrt{%
kh-u^{\prime}}}
\int}{}\times
\biggl( \sqrt{%
kh-u^{\prime}}+\frac{\llvert  S(\xi)-S(x)\rrvert }{\sqrt{%
kh-u^{\prime}}} \biggr)
\nonumber
\\
&&\hphantom{{}\times\int_{kh/2}^{kh}\frac{\mathrm{d}u^{\prime}}{\sqrt{%
kh-u^{\prime}}}
\int} {}\times \biggl( 1+\sqrt{u^{\prime}}+\frac{\llvert  S(y)-S(\xi
)\rrvert }{\sqrt{u^{\prime}}}
\biggr)^{2}\,\mathrm{d}\xi
\nonumber
\\
&\leq& \frac{C}{kh}%
\exp[2Mkh] \frac{\exp [ H(y)-H(x) ] }{\sqrt{2\uppi kh}\sigma(y)}\exp \biggl(
-\frac{(S(y)-S(x))^{2}}{2kh} \biggr)
\nonumber
\\
&&{} \times \int_{kh/2}^{kh}\frac{\mathrm{d}u^{\prime}}{\sqrt{kh-u^{\prime}}}\int
\frac
{\mathrm{d}z%
}{\sqrt{2\uppi}}\exp \biggl( -\frac{z^{2}}{2} \biggr)
\\
&&{} \times \biggl( 1+\sqrt{u^{\prime}}+\sqrt{\frac{u^{\prime}}{%
kh}}
\frac{\llvert  S(y)-S(x)\rrvert }{\sqrt{kh}}+\llvert z\rrvert \sqrt{\frac{kh-u^{\prime}}{kh}}
\biggr)^{2}
\nonumber
\\
&&{} \times \biggl( \sqrt{kh-u^{\prime}}+\sqrt{\frac{kh-u^{\prime
}}{kh}}
\frac{%
\llvert  S(y)-S(x)\rrvert }{\sqrt{kh}}+\llvert z\rrvert \sqrt{%
\frac{u^{\prime}}{kh}}
\biggr)
\nonumber
\\
&\leq& \frac{C}{kh}\exp[2Mkh]\widehat{p}(kh,x,y)\int_{kh/2}^{kh}%
\frac{\mathrm{d}u^{\prime}}{\sqrt{kh-u^{\prime}}}\int\frac{\mathrm{d}z}{\sqrt{2\uppi
}}\exp \biggl( -\frac{z^{2}}{2}
\biggr)
\nonumber
\\
&& {}\times \biggl( 1+\sqrt{kh}+\frac{\llvert  S(y)-S(x)\rrvert
}{\sqrt{kh}}%
+\llvert z\rrvert
\biggr)^{2} \biggl( \sqrt{kh}+\frac{\llvert
S(y)-S(x)\rrvert }{\sqrt{kh}}+\llvert z\rrvert
\biggr)
\nonumber
\\
&\leq& \frac{C}{\sqrt{kh}}\exp[2Mkh]\widehat{p}(kh,x,y)\int\frac
{\mathrm{d}z%
}{\sqrt{2\uppi}}
\exp \biggl( -\frac{z^{2}}{2} \biggr) \biggl( 1+\sqrt{kh}+
\frac{\llvert  S(y)-S(x)\rrvert }{\sqrt {kh}}%
+\llvert z\rrvert \biggr)^{3}
\nonumber
\\
&\leq& \frac{C}{\sqrt{kh}}%
p(kh,x,y) \biggl( 1+\frac{\llvert  S(y)-S(x)\rrvert }{\sqrt{kh}}
\biggr)^{3}.\nonumber
\end{eqnarray}
By similar calculations, we obtain that
%
\begin{equation}\label{a13}
\llvert \Im_{2}\rrvert \leq\frac{C}{\sqrt{kh}}p(kh,x,y)%
\biggl( 1+\frac{\llvert  S(y)-S(x)\rrvert }{\sqrt{kh}} \biggr)^{3}.
\end{equation}
The lemma now follows from our assumptions on $\sigma,$ (\ref%
{a5}), (\ref{a12}) and (\ref{a13}).
\end{pf}

We will also make use of the following bound.
%
\begin{lemma}\label{lemma4aa}
For any polynomials $P_l(x)$ and $P_m(x)
$ of degrees $l$ and $m$, there exists a constant $C$, depending only
on $l$, $m$ and the coefficients of the polynomials, such that
uniformly for $w \in[0, kh/4)$ the following inequalities hold
\begin{eqnarray*}
&&\int_0^{kh/4} u^{-1/2} \int p(u,x,z)
P_l \biggl(\biggl\llvert \frac
{z-x}{\sqrt{u}}\biggr\rrvert \biggr)
p(kh -w-u,z,y) P_m \biggl(\biggl\llvert \frac
{y-z}{\sqrt{kh -w-u}}\biggr
\rrvert \biggr) \,\mathrm{d}z \,\mathrm{d}u
\\
&&\quad  \leq C \sqrt{kh} p(kh-w,x,y) \biggl( 1 + \biggl\llvert \frac{y-x}{\sqrt{kh -w}}
\biggr\rrvert^{l+m} \biggr).
\end{eqnarray*}
\end{lemma}
\begin{pf}
These bounds can be
easily shown by using the representation
(\ref{a2}) and calculations of similar convolution integrals for
Gaussian densities.
\end{pf}

Put
\begin{eqnarray}
\delta_{2}(x,y)=h\frac{\pi_{2}(kh,x,y)}{p(kh,x,y)}.
\nonumber
\end{eqnarray}
We now state a bound for $\delta_{2}(x,y)$.
%
\begin{lemma}\label{lemma4b}
There exists a constant C such that for
$x,y \in\mathbb{R}$
\[
\bigl\llvert \delta_{2}(x,y)\bigr\rrvert \leq\frac{C}{k}
\biggl[ 1+ \biggl( \frac{%
\llvert  y-x\rrvert }{\sqrt{kh}} \biggr)^{7} \biggr].
\]
\end{lemma}
\begin{pf}
Note that the function $\pi_{2}(kh,x,y)$ can
be written as%
\[
\pi_{2}(kh,x,y)=\Im_{3} + \Im_{4},
\]
where
\[
\Im_{3} = \sum_{i=1}^{4}\int
_{0}^{kh}\mathrm{d}u\int p(u,x,\xi)f_{i}(
\xi)%
\frac{\partial^{i}}{\partial\xi^{i}}p(kh-u,\xi,y)\,\mathrm{d}\xi,
\]
with $f_{4}(\xi)=\mu_{4}(\xi)-3\sigma^{4}(\xi)$ and $f_{i}(\xi
),i=1,2,3,$ depending on the coefficients of the operator $L$ and their
derivatives up to the order 2. Furthermore, the term $\Im_{4}$ is
defined as
\begin{eqnarray*}
\Im_{4} &=& { 1 \over36}\int_{u+w \leq kh; u,w \geq0} p
\bigl(u,x, \xi^*\bigr) \mu_3\bigl(\xi^*\bigr)\frac{\partial^{3}}{(\partial\xi^*)^{3}} p
\bigl(kh-u-w, \xi^*, \xi\bigr) \mu_3(\xi) \\
&&\hphantom{{ 1 \over36}\int_{u+w \leq kh; u,w \geq0}}{}\times \frac{\partial^{3}}{(\partial\xi)^{3}} p(w,
\xi, y) \,\mathrm{d}\xi \,\mathrm{d} \xi^* \,\mathrm{d}u \,\mathrm{d}w.
\end{eqnarray*}

Applying the same arguments as in the proof
of Lemma~\ref{lemma4}, we get
\[
\Biggl\llvert \sum_{i=1}^{3}\int
_{0}^{kh}\mathrm{d}u\int p(u,x,\xi)f_{i}(\xi )
\frac{%
\partial^{i}}{\partial\xi^{i}}p(kh-u,\xi,y)\,\mathrm{d}\xi\Biggr\rrvert \leq \frac{C%
}{\sqrt{kh}}p(kh,x,y)
\biggl( 1+\frac{\llvert  y-x\rrvert
}{\sqrt{kh}}%
\biggr)^{3}.
\]
For $i=4,$ we have to estimate the integral%
\[
\int_{0}^{kh}\mathrm{d}u\int p(u,x,\xi)f_{4}(
\xi)\frac{\partial^{4}}{\partial\xi^{4}}p(kh-u,\xi,y)\,\mathrm{d}\xi.
\]
With calculations very similar to the ones used in the proof of Lemma~\ref{lemma4} we get%
%
\begin{equation}
\label{claimaa} \llvert \Im_{3}\rrvert \leq\frac{C}{kh}p(kh,x,y)
\biggl[ 1+ \biggl( \frac{\llvert  y-x\rrvert }{\sqrt{kh}} \biggr)^{4} \biggr] .
\end{equation}
It remains to bound $\Im_{4}$. We write
\[
\Im_{4} = \Im_{4a} + \Im_{4b}+
\Im_{4c},
\]
where
\begin{eqnarray*}
\Im_{4a} &=& { 1 \over36}\int_{I_a }
\cdots \,\mathrm{d}\xi \,\mathrm{d} \xi^* \,\mathrm{d}u \,\mathrm{d}w,
\\
\Im_{4b} &=& { 1 \over36}\int_{I_b }
\cdots \,\mathrm{d}\xi \,\mathrm{d} \xi^* \,\mathrm{d}u \,\mathrm{d}w,
\\
\Im_{4c} &=& { 1 \over36}\int_{I_c}
\cdots \,\mathrm{d}\xi \,\mathrm{d} \xi^* \,\mathrm{d}u \,\mathrm{d}w,
\\
I_a&=& \bigl\{ \bigl(u,w,\xi, \xi^*\bigr)\dvtx  u,w,\xi, \xi^* \in
\mathbb{R}; u+w \leq kh; 0 \leq u; kh/4 \leq w \bigr\},
\\
I_b&=& \bigl\{ \bigl(u,w,\xi, \xi^*\bigr)\dvtx  u,w,\xi, \xi^* \in
\mathbb{R}; u+w \leq kh; kh/4 \leq u; 0 \leq w < kh/4\bigr\},
\\
I_c&=& \bigl\{ \bigl(u,w,\xi, \xi^*\bigr)\dvtx  u,w,\xi, \xi^* \in
\mathbb{R}; u+w \leq kh; 0 \leq u < kh/4; 0 \leq w < kh/4 \bigr\}.
\end{eqnarray*}
We now show that for some constant $C>0$
%
\begin{equation}
\label{claimbbc} \llvert \Im_{4c}\rrvert \leq\frac{C}{kh}p(kh,x,y)
\biggl[ 1+ \biggl( \frac{\llvert  y-x\rrvert }{\sqrt{kh}} \biggr)^{4} \biggr] .
\end{equation}
For this estimate one applies the following bound that follows by
partial integration:
\begin{eqnarray*}
|\Im_{4c} |&\leq& { 1 \over36}\biggl\llvert \int
_{I_c} \frac{\partial
}{\partial\xi^*} \bigl[p\bigl(u,x, \xi^*\bigr)
\mu_3\bigl(\xi^*\bigr)\bigr]\frac{\partial^{4}}{(\partial\xi^*)^{2}(\partial\xi)^{2}} p\bigl(kh-u-w, \xi^*,
\xi\bigr) \mu_3(\xi) \\
&&\hphantom{{ 1 \over36}\biggl\llvert \int
_{I_c}}{}\times \frac{\partial}{\partial\xi} p(w, \xi, y) \,\mathrm{d}\xi \,\mathrm{d} \xi^* \,\mathrm{d}u \,\mathrm{d}w\biggr\rrvert .
\end{eqnarray*}
The integrand can be bounded with the help of (\ref{a6a}), (\ref{a6})
and (\ref{a7c}). Because of the bounds of Lemma~\ref{lemma4aa} this
implies (\ref{claimbbc}).

To bound $\Im_{4a}$ we use that:
\begin{eqnarray*}
36 \llvert \Im_{4a} \rrvert &\leq& \biggl\llvert \int
_{I_a} \frac
{\partial}{\partial\xi^*} \bigl[p\bigl(u,x, \xi^*\bigr)
\mu_3(\xi)\bigr] \biggl[ \frac{\partial^{2}}{(\partial\xi^*)^{2}} - \frac{\partial^{2}}{(\partial\xi^*)(\partial\xi)}
\biggr] p\bigl(kh-u-w, \xi^*, \xi\bigr) \mu_3(\xi)
\\
&&\hphantom{\biggl\llvert \int
_{I_a}}{}  \times\frac{\partial^3}{(\partial\xi)^3} p(w, \xi, y) \,\mathrm{d}\xi \,\mathrm{d} \xi^* \,\mathrm{d}u \,\mathrm{d}w\biggr
\rrvert
\\
&& {}+ \biggl\llvert \int_{I_a} \frac{\partial}{\partial\xi^*} \bigl[p
\bigl(u,x, \xi^*\bigr) \mu_3(\xi)\bigr] \frac{\partial}{\partial\xi^*} p
\bigl(kh-u-w, \xi^*, \xi\bigr) \frac{\partial}{\partial\xi} \mu_3(\xi)
\\
&&\hphantom{{}+ \biggl\llvert \int_{I_a}}{}  \times\frac{\partial^3}{(\partial\xi)^3} p(w, \xi, y) \,\mathrm{d}\xi \,\mathrm{d} \xi^* \,\mathrm{d}u \,\mathrm{d}w\biggr
\rrvert
\\
&& {}+ \biggl\llvert \int_{I_a} \frac{\partial}{\partial\xi^*} \bigl[p
\bigl(u,x, \xi^*\bigr) \mu_3(\xi)\bigr] \frac{\partial^2}{(\partial\xi^*)(\partial\xi)} \bigl[ p
\bigl(kh-u-w, \xi^*, \xi\bigr) \mu_3(\xi)\bigr]
\\
&&\hphantom{{}+ \biggl\llvert \int_{I_a}}{}  \times\frac{\partial^3}{(\partial\xi)^3} p(w, \xi, y) \,\mathrm{d}\xi \,\mathrm{d} \xi^* \,\mathrm{d}u \,\mathrm{d}w\biggr
\rrvert
\\
&=& \Im_{4aa}+ \Im_{4ab}+ \Im_{4ac}.
\end{eqnarray*}
These terms can be easily bounded by using the bounds of Lemma~\ref
{lemma5}. Because of the bounds of Lemma~\ref{lemma4aa}, this implies
%
\begin{equation}
\label{claimbba} \llvert \Im_{4a}\rrvert \leq\frac{C}{kh}p(kh,x,y)
\biggl[ 1+ \biggl( \frac{\llvert  y-x\rrvert }{\sqrt{kh}} \biggr)^{7} \biggr] .
\end{equation}
To get a bound for $\Im_{4ac}$ we use that by partial integration:
\begin{eqnarray*}
\Im_{4ac} &=& \biggl\llvert \int_{I_a}
\frac{\partial}{\partial\xi^*} \bigl[p\bigl(u,x, \xi^*\bigr) \mu_3(\xi)\bigr]
\frac{\partial}{\partial\xi^*} \bigl[ p\bigl(kh-u-w, \xi^*, \xi\bigr) \mu_3(\xi)
\bigr]
\\
&& \hphantom{\biggl\llvert \int_{I_a}}{} \times\frac{\partial^4}{(\partial\xi)^4} p(w, \xi, y) \,\mathrm{d}\xi \,\mathrm{d} \xi^* \,\mathrm{d}u \,\mathrm{d}w\biggr
\rrvert .
\end{eqnarray*}
Similarly one shows that
%
\begin{equation}
\label{claimbbb} \llvert \Im_{4b}\rrvert \leq\frac{C}{kh}p(kh,x,y)
\biggl[ 1+ \biggl( \frac{\llvert  y-x\rrvert }{\sqrt{kh}} \biggr)^{7} \biggr] .
\end{equation}
The statement of Lemma~\ref{lemma4b} follows now from (\ref
{claimaa}), (\ref{claimbba}), (\ref{claimbbb}) and (\ref{claimbbc}).
\end{pf}

\section*{Acknowledgements}
This study was carried out within ``The National Research University
Higher School of Economics''
Academic Fund Program in 2012--2013, research Grant 11-01-0083.
Support by Grant 436RUS113/467/81-2
from the Deutsche Forschungsgemeinschaft is also acknowledged. The
research of Enno Mammen and Jeannette
Woerner was supported by the German Science Foundation (DFG) in the
framework of the German--Swiss Research
Group FOR 916 ``Statistical Regularization and Qualitative Constraints''.



\printhistory

\end{document}